\documentclass[11pt,reqno]{amsart}
\usepackage{amssymb,amsmath,amsthm}
\usepackage{amscd}
\input amssym.def
\input amssym
\newtheorem{theorem}{Theorem}[section]

\newtheorem{remark}{Remark}[section]

\newtheorem{proposition}{Proposition}[section]

\newcommand{\be}{\begin{equation}}
\newcommand{\ee}{\end{equation}}

\renewcommand{\theequation}{\thesection.\arabic{equation}}
\renewcommand{\thetheorem}{\thesection.\arabic{theorem}}
\renewcommand{\theequation}{\thesection.\arabic{equation}}
\begin{document}

\title[]
{A motivated proof of Gordon's identities}

\author{James Lepowsky and Minxian Zhu}

\begin{abstract}

We generalize the ``motivated proof'' of the Rogers-Ramanujan
identities given by Andrews and Baxter to provide an analogous
``motivated proof'' of Gordon's generalization of the Rogers-Ramanujan
identities.  Our main purpose is to provide insight into certain
vertex-algebraic structure being developed.

\end{abstract}

\maketitle

\renewcommand{\theequation}{\thesection.\arabic{equation}}
\renewcommand{\thetheorem}{\thesection.\arabic{theorem}}
\setcounter{equation}{0} \setcounter{theorem}{0}
\setcounter{section}{0}

\section{Introduction}
\setcounter{equation}{0}

In \cite{AB}, G. Andrews and A. Baxter have provided an interesting
``motivated proof'' of the two Rogers-Ramanujan identities (among the
large number of proofs in the literature), which we write in the form:
\begin{eqnarray*}
\prod_{n \geq 1, \,\, n \not\equiv 0, \, \pm 2 \, (\text{mod } 5)}
\frac{1}{1-q^n}
& = &
\sum_{m \geq 0} p_1 (m)q^m
\end{eqnarray*}
and
\begin{eqnarray*}
\prod_{n \geq 1, \,\, n \not\equiv 0, \, \pm 1 \, (\text{mod } 5)}
\frac{1}{1-q^n}
& = &
\sum_{m \geq 0} p_2 (m)q^m,
\end{eqnarray*}
where
\begin{eqnarray*}
p_1 (m)
& = & 
\text{the number of partitions of $m$ for which adjacent parts have}
\nonumber \\
& &
\quad \text{difference at least 2}
\end{eqnarray*}
and
\begin{eqnarray*}
p_2 (m)
& = & 
\text{the number of partitions of $m$ for which adjacent parts have}
\nonumber \\
& & 
\quad \text{difference at least 2 and in which 1 does not appear,}
\end{eqnarray*}
and where $q$ is a formal variable.  The idea was to start from the
product sides, in which the difference-two condition is invisible, and
to both motivate the expressions on the right-hand sides and prove the
two identities, as follows:

One subtracts the second product side, called $G_2 (q)$, from the
first one, called $G_1 (q)$, then one divides the difference by $q$,
giving a formal series $G_3 (q)$, and then one forms $G_4 (q)=(G_2 (q)
- G_3 (q))/q^2$.  One repeats this process, giving $G_{i} (q)=(G_{i-2}
(q) - G_{i-1} (q))/q^{i-2}$ for all $i \geq 3$.  One notices
empirically that for each $i \geq 1$, $G_i (q)$ is a formal power
series (that is, it involves only nonnegative powers of $q$), it has
constant term $1$, and $G_i (q) - 1$ is divisible by $q^i$.  This is
the ``Empirical Hypothesis'' of Andrews-Baxter.  Assuming its truth,
one easily gets the two Rogers-Ramanujan identities.  Then, with this
as motivation, one proceeds to prove the Emprical Hypothesis directly
{}from the product sides, thus proving the Rogers-Ramanujan
identities.  (In \cite{AB}, $q$ is taken to be a complex variable of
absolute value less than 1, but in fact, the content of the argument
is purely formal, and we shall take $q$ to be a formal variable.)

An initial motivation for the work in \cite{AB} was to show {\it from
the product sides} the highly non-obvious fact that the difference of
the two product sides (in the same order as above) is a formal power
series with nonnegative coefficients.  This was Leon Eherenpreis's
problem, and Andrews and Baxter gave a motivated proof of this fact as
preparation for their motivated proof of the identities themselves.
Also, as is recalled in \cite{AB}, that proof of the identities is
closely related to Baxter's proof of the identities, starting in
\cite{B}, and moreover, that proof is also closely related to Rogers's
and Ramanujan's proof in \cite{RR}.

In the present paper, we generalize the Andrews-Baxter ``motivated
proof'' to give an analogous proof of Gordon's form (\cite{G},
\cite{A1}) of the Gordon-Andrews generalizations of the
Rogers-Ramanujan identities, essentially in the form in which they are
presented in Theorem 7.5 of \cite{A2}, where for each $k \geq 2$ and
$i=1,\dots, k$, a suitable infinite product in $q$ is expressed as a
formal power series in $q$ for which the coefficient of $q^m$ is the
number of partitions of $m$ such that parts at distance $k-1$ have
difference at least 2 and such that 1 appears at most $k-i$ times; the
case $k=2$ is the pair of Rogers-Ramanujan identities.  (For a
partition $m=p_1 + \cdots + p_n$ of $m$ with $p_1 \geq \cdots \geq p_n
> 0$ and for $t \geq 1$, saying that parts at distance $t$ have
difference at least 2 means that $p_s - p_{s+t} \geq 2$ whenever $s
\geq 1$ and $s + t \leq n$.)  We do not address Andrews's multisum
form of the sum sides, as presented in Theorem 7.8 of \cite{A2}.
While our proof is (necessarily) more complicated than the case $k=2$,
it is similar, although interesting new phenomena arise.  Also, since
we know what is going to happen, we take the liberty of identifying
the appropriate analogue and generalization of the Empirical
Hypothesis as our ``Empirical Hypothesis,'' even though we did not
observe its validity empirically before actually proving it directly
{}from the product sides.  Our proof includes (a variant of) the proof
in \cite{AB} as a special case.

Our reason for wanting to work out such a proof stemmed from the
relations between the Rogers-Ramanujan identities, and
generalizations, and what is now known to be vertex operator algebra
theory, as follows:

By retrospective analogy with the approach to the Rogers-Ramanujan
identities in \cite{AB}, the vertex-operator-theoretic proof of the
Rogers-Ramanujan identities along with the vertex-operator-theoretic
interpretation of their Gordon-Andrews-Bressoud generalizations in
\cite{LW2}--\cite{LW4} also started from the product sides as
``given'' \cite{LM}, and the problem at the time was to discover what
turned out to be new structure, not previously anticipated, that would
``explain'' the sum sides (the difference-two condition and its
variants).  The result, based on the vertex operator theory whose
discovery and development was the subject of those works, was the
theory and application of ``$Z$-algebras''---primarily {\it twisted}
$Z$-algebras in \cite{LW2}--\cite{LW4}, and then {\it untwisted}
$Z$-algebras in \cite{LP},---which later turned out to be understood
in retrospect as the natural generating substructures of certain
generalized vertex algebras, or abelian intertwining algebras, as
developed in \cite{DL}, and twisted modules for them.  (These
$Z$-algebras, both untwisted and twisted, were also to arise as
``parafermion algebras'' in conformal field theory (\cite{ZF1},
\cite{ZF2}).)  In \cite{LW2}--\cite{LW4}, each identity was related to
certain vertex-operator-theoretic structure constructed from a certain
module for an affine Lie algebra, and the structures associated with
different identities were not ``compared'' with one another, in the
spirit of product sides being subtracted, etc.; the structures for the
different identities were developed in parallel (with the proofs for
the parallel structures certainly being closely related).  The
structures were based on twisted $Z$-operators, built starting from
the twisted vertex operator in \cite{LW1}.  Later, a very different
vertex-algebraic approach to the sum sides of the Rogers-Ramanujan and
Gordon-Andrews identities was developed in \cite{CLM1}, \cite{CLM2},
\cite{CalLM1} and \cite{CalLM2}, this time based on {\it untwisted
intertwining} operators (in the sense of vertex operator algebra
theory), and this time, indeed relating the family of different
identities with the same ``modulus'' (the number $2k+1$ in the
notation above).  In this work, the classical Rogers-Ramanujan
recursion ($q$-difference equation) and Rogers-Selberg recursions had
suggested what turned out to become certain systems of exact sequences
constructed from untwisted intertwining operators among the
``principal subspaces,'' in the sense of \cite{FS1}--\cite{FS2}, of
different modules for certain vertex algebras, and it was this
vertex-algebraic structure that was of primary interest.

With this as background, we can now say that the initial reason for
our interest in the motivated proof in \cite{AB} is that that proof
suggested to one of us (J.\,L.)  and Antun Milas the potential new
idea in vertex operator algebra theory to use {\it twisted}
intertwining operators among {\it twisted} modules for suitable
vertex-algebraic structures to develop new structure in the theory
that would ``re-explain'' the identities from this new point of view.
Such a program is underway.  As was the case in the work mentioned
above, the potential new structure suggested, this time, by the
``motivated proof'' is our main goal.

In Section 2 we give our ``motivated proof'' of Gordon's identities,
in Section 3 we reinterpret the sequence of equalities (\ref{Gi}) at
the core of the proof, and in Section 4 we ``explain'' the meaning of
this sequence of equalities.

The series in $q$ and $z$ below are formal series (rather than
convergent series in complex variables).

\section{The motivated proof}
\setcounter{equation}{0}

Fix an integer $k \geq 2$.  For each $i = 1,\dots,k$, define
\begin{eqnarray}\label{Gidef}
G_i & = &
\prod_{n \geq 1, \,\, n \not\equiv 0, \, \pm (k +1 -i) \, (\text{mod } 2k+1)}
\frac{1}{1-q^n}.
\end{eqnarray}
Recalling the Jacobi triple product identity,
\begin{eqnarray*}
\sum_{\lambda \in \mathbb{Z}} (-1)^\lambda z^\lambda q^{\lambda^2} 
& = &
\prod_{n \geq 0} (1-q^{2n+2})(1- z q^{2n+1})(1-z^{-1} q^{2n+1}),
\end{eqnarray*}
and replacing $q$ by $q^{\frac{2k+1}{2}}$ and $z$ by
$q^{\frac{2i-1}{2}}$, we have
\begin{eqnarray}\label{Gk1i}
G_i
& = &
\frac{ 1+ \sum_{\lambda \geq 1}
(-1)^\lambda q^{ (2k+1) {\lambda \choose 2} + (k-i+1) \lambda} (1 + q^{(2i-1) \lambda })
}{ \prod_{n \geq 1} (1-q^n)} 
\\
\label{Gk1i2}
& = &
\frac{ \sum_{\lambda \geq 0}
(-1)^\lambda q^{ (2k+1) {\lambda \choose 2} + (k+i) \lambda} (1 - q^{ (k-i+1) (2 \lambda+1)} )}
{ \prod_{n \geq 1} (1-q^n)} 
\end{eqnarray}
for $i=1, \dots, k$.

Define $k-1$ further formal series $G_{k + 1}, \dots, G_{2k - 1}$ by
\begin{eqnarray} 
G_{k -1 +i} & = & \frac{G_{k -i +1 } - G_{k- i + 2}}{q^{i-1} }
\end{eqnarray}
for $i = 2, \dots, k$.  Then for these new series, by (\ref{Gk1i}) we
have
\begin{eqnarray}\label{Gformulaj=1}
& G_{k -1 +i} &
\nonumber \\
& = & \frac{
\sum_{\lambda \geq 1}
(-1)^\lambda q^{ (2k+1) {\lambda \choose 2} + i \lambda} (1 + q^{(2k -2i +1) \lambda })
-  \sum_{\lambda \geq 1}
(-1)^\lambda q^{ (2k+1) {\lambda \choose 2} + (i-1) \lambda} (1 + q^{(2k - 2i +3) \lambda })
}{q^{i-1}  \prod_{n \geq 1} (1-q^n)}
\nonumber \\
& = & \frac{
 \sum_{\lambda \geq 1}
(-1)^{ \lambda}  q^{ (2k+1) {\lambda \choose 2} + (i -1) \lambda}
[ q^\lambda (1 + q^{ (2k -2i +1) \lambda} ) - (1 + q^{ (2k - 2i +3 ) \lambda}) ]
}{q^{i -1}  \prod_{n \geq 1} (1-q^n)}
\nonumber \\
& = & \frac{
 \sum_{\lambda \geq 1}
(-1)^{ \lambda +1 }  q^{ (2k+1) {\lambda \choose 2} + ( i -1) \lambda}
(1 - q^\lambda) (1 - q^{ (2k -2i +2) \lambda} )
}{ q^{i-1}  \prod_{n \geq 1} (1-q^n)}
\nonumber \\
& = & \frac{
 \sum_{\lambda \geq 0}
(-1)^{ \lambda }  q^{ (2k+1) {\lambda \choose 2} + (2k + i ) \lambda}
(1 - q^{\lambda+1}) ( 1 - q^{ (k - i + 1 ) (2\lambda + 2) } )
}{ \prod_{n \geq 1} (1-q^n) }
\end{eqnarray}
for $i=2, \dots, k$.  In particular, $G_{k -1 +i} \in \mathbb{C}[[q]]$
(that is, $G_{k -1 +i}$ is a formal power series), and its constant
term is $1$.

Moreover, (\ref{Gformulaj=1}) remains valid for $i=1$ as well, since
the right-hand side for $i=1$ agrees with (\ref{Gk1i2}) for $i=k$.
Indeed, when $i=1$, the right-hand side of (\ref{Gformulaj=1}) is
\begin{eqnarray}
\frac{
\sum_{\lambda \geq 0}
(-1)^{ \lambda }  q^{ (2k+1) {\lambda \choose 2} + (2k + 1) \lambda}
(1 - q^{\lambda+1}) ( 1 - q^{ 2k (\lambda +1) } )
}{ \prod_{n \geq 1} (1-q^n) }.
\nonumber
\end{eqnarray}
Breaking up the last factor in the numerator, we obtain two terms, the
first of which can be rewritten as
\begin{eqnarray}
\frac{
\sum_{\lambda \geq 0}
(-1)^{ \lambda }  q^{ (2k+1) {\lambda \choose 2} + 2k \lambda}
(q^\lambda - q^{2 \lambda+1})
}{ \prod_{n \geq 1} (1-q^n) }.
\nonumber
\end{eqnarray}
After the re-indexing $\lambda \to \lambda-1$, the second term becomes
\begin{eqnarray}
\frac{
\sum_{\lambda \geq 1}
(-1)^{ \lambda }  q^{ (2k+1) {\lambda-1 \choose 2} + (2k + 1) (\lambda-1)
+ 2k \lambda}
(1 - q^{\lambda})
}{ \prod_{n \geq 1} (1-q^n) }.
\nonumber
\end{eqnarray}
Noting that $ {\lambda-1 \choose 2} + (\lambda-1) = {\lambda \choose
2}$ and that allowing $\lambda = 0$ in the numerator only results in
adding zero, we can combine the two terms to obtain (\ref{Gk1i2}) for
$i = k$.  That is, (\ref{Gformulaj=1}) holds for all $i =1, 2, \dots,
k$.

In general, for $j \geq 1$ and $i = 2, \dots, k$, define the formal series
\begin{eqnarray}\label{Gjk-1i}
G_{ (k-1)j + i} & = & 
\frac{ G_{ (k-1)(j-1) + k- i +1 } - G_{ (k-1)(j-1) + k - i +2 }}{q^{ (i-1)j}}. 
\end{eqnarray}

\begin{theorem}\label{G}
For $j \geq 0$ and $i = 1, \dots, k$, $G_{(k -1)j +i} \in
\mathbb{C}[[q]]$ and in fact

\begin{eqnarray}\label{Gformula}
& G_{(k-1) j + i} & 
\nonumber \\
& = & \frac{
\sum_{\lambda \geq 0} (-1)^\lambda 
q^{ (2k +1) {\lambda \choose 2} + [ k (j+1) + i ] \lambda }
(1- q^{\lambda +1} ) \cdots (1 - q^{\lambda + j} )
(1 - q^{ ( k - i + 1) (2 \lambda + j + 1) } )
}{\prod_{n \geq 1} (1-q^n) }. \nonumber\\
\end{eqnarray}
In particular, denoting the right-hand side of (\ref{Gformula}) by
$H_{(k-1) j + i}$, we have that for each $j \geq 1$, the two
expressions for $G_{kj - j + 1}$ given by (\ref{Gformula}) are equal:
\begin{eqnarray}\label{H=H}
H_{(k-1) j + 1} & = & H_{(k-1) (j-1) + k}.
\end{eqnarray}
\end{theorem}

\begin{proof}
By (\ref{Gk1i2}), (\ref{Gformula}) holds for $j=0$.  By
(\ref{Gformulaj=1}) and the above, (\ref{Gformula}) and (\ref{H=H})
both hold for $j=1$.  Take $j \geq 1$.  Suppose that (\ref{Gformula})
holds for $G_{ (k-1)j + q}$, $1 \leq q \leq k$.  We will show that it
holds for $G_{(k-1) (j +1) + i}$, $1 \leq i \leq k$.

First let $i=2, \dots, k$.  By the recursion (\ref{Gjk-1i}), we have
\begin{eqnarray}
\lefteqn{G_{(k-1) (j +1) + i} \, = \,
\frac{ G_{ (k-1)j + k- i +1 } - G_{ (k-1)j + k - i +2 } } {q^{ (i-1)(j+1)} }}
\nonumber \\
& = &
\frac{
\sum_{\lambda \geq 0} (-1)^\lambda 
q^{ (2k+1) {\lambda \choose 2} + [ k (j+1) +k - i +1 ] \lambda}
(1- q^{\lambda +1} ) \cdots (1- q^{\lambda + j})
(1- q^{ i (2 \lambda + j +1)} )
}{ q^{ (i-1)(j+1) } \prod_{n \geq 1} (1-q^n) }
\nonumber \\
& & - \, \frac{
\sum_{\lambda \geq 0} (-1)^\lambda 
q^{ (2k+1) {\lambda \choose 2} + [ k (j+1) +k - i + 2 ] \lambda}
(1- q^{\lambda +1} ) \cdots (1- q^{\lambda + j}) 
(1- q^{ ( i -1) (2 \lambda + j +1)} )
}{q^{ (i-1)(j+1) }  \prod_{n \geq 1} (1-q^n) } 
\nonumber
\end{eqnarray}
We use the last factor in each of the two numerators to split each of
the two summations into two sums, $\sum_1 = \sum_{11} - \sum_{12}$, $
\sum_2 = \sum_{21} - \sum_{22}$, and we combine the terms in a
different way: $( \sum_{11} - \sum_{21} ) + ( - \sum_{12} + \sum_{22}
)$.

Now $\sum_{11} - \sum_{21}$ equals
\begin{eqnarray}\label{Sig11-21}
& & \frac{ 
\sum_{\lambda \geq 0} (-1)^\lambda 
q^{ (2k+1) {\lambda \choose 2} + [ k (j+1) +k - i +1 ] \lambda}
(1 - q^\lambda) (1- q^{\lambda +1} ) \cdots (1- q^{\lambda + j})
}{ q^{ (i-1)(j+1) } \prod_{n \geq 1} (1-q^n) }. 
\end{eqnarray}
Note that the $\lambda = 0$ term in (\ref{Sig11-21}) vanishes, so the
summation is actually over $\lambda \geq 1$.  Making the index change
$\lambda \to \lambda +1$ we obtain
\begin{eqnarray}\label{Sig11-21new}
\lefteqn{\frac{ \sum_{\lambda \geq 0} (-1)^{\lambda +1 }
q^{ (2k+1) {\lambda \choose 2} + [ k ( j+ 2) + 2 k -  i + 2 ] \lambda + k ( j+1) + k - i +1 }
(1- q^{\lambda +1} ) \cdots (1 - q^{\lambda + j +1} )
}{ q ^{ (i-1)(j+1) } \prod_{n \geq 1} (1-q^n) }}
\nonumber \\
& = & 
\frac{ \sum_{\lambda \geq 0} (-1)^{\lambda +1 }
q^{ (2k+1) {\lambda \choose 2} + [ k ( j+2) + i ] \lambda +  (k - i +1) (j+2) }
(1- q^{\lambda +1} ) \cdots (1 - q^{\lambda + j +1} )
q^{ (2k - 2i +2 ) \lambda }
}{ \prod_{n \geq 1} (1-q^n) }.
\nonumber\\
\end{eqnarray}
Similarly, 
$- \sum_{12} + \sum_{22} $
equals 
\begin{eqnarray}\label{-Sig12+22}
\lefteqn{\frac{
\sum_{\lambda \geq 0} (-1)^\lambda
q^{ (2k+1) {\lambda \choose 2} + [ k(j+1) + k +i ] \lambda + (i-1) (j+1) }
(1- q^{\lambda+1}) \cdots (1 - q^{\lambda +j }) 
( - q^{ \lambda + j +1 } + 1) 
}{ q ^{ (i-1)(j+1) } \prod_{n \geq 1} (1-q^n) }}
\nonumber\\
& = & 
\frac{
\sum_{\lambda \geq 0} (-1)^\lambda
q^{ (2k+1) {\lambda \choose 2} + [ k(j+2) +i ] \lambda }
(1- q^{\lambda+1}) \cdots (1 - q^{\lambda +j }) 
( 1 - q^{ \lambda + j +1 } ) 
}{ \prod_{n \geq 1} (1-q^n) }. 
\nonumber\\
\end{eqnarray}
Combining (\ref{Sig11-21new}) and (\ref{-Sig12+22}) we get 
\begin{eqnarray}
G_{(k-1) (j +1) + i}
& = &
\frac{
\sum_{\lambda \geq 0} (-1)^\lambda
q^{ (2k+1) {\lambda \choose 2} + [ k(j+2) +i ] \lambda }
(1- q^{\lambda+1}) \cdots
( 1 - q^{ \lambda + j +1 } ) 
( 1 - q^{ (k-i+1) ( 2 \lambda +j +2) } )
}{ \prod_{n \geq 1} (1-q^n) },
\nonumber 
\end{eqnarray}
proving (\ref{Gformula}) for the case $j+1$ and $i = 2, \dots, k$.

For the case $j+1$ and $i = 1$, we observe that (\ref{Gformula})
follows from the induction hypothesis and (\ref{H=H}) for the case
$j+1$, and (\ref{H=H}) in turn follows (for any $j$) by virtually the
same argument as the one above for $j=1$.
\end{proof}

Theorem \ref{G} implies that for $j \geq 0$,
\begin{eqnarray}
\lefteqn{G_{(k-1) j + i}} 
\nonumber\\
& = & 
\frac{ 1 - q^{ ( k-i + 1) ( j+1) } }
{ (1 - q^{ j+1} ) ( 1- q^{ j+2} ) \cdots}
\nonumber  \\
& & 
+ \, \frac{
\sum_{\lambda \geq 1} (-1)^\lambda 
q^{ (2k +1) {\lambda \choose 2} + [ k (j+1) + i ] \lambda }
(1- q^{\lambda +1} ) \cdots (1 - q^{\lambda + j} )
(1 - q^{ ( k - i + 1) (2 \lambda + j + 1) } )
}{ \prod_{n \geq 1} (1-q^n) }. 
\nonumber \\
& = & 
1 + q^{j+1} \gamma_i^{(j+1)} (q)  \quad \text{ if  } 1  \leq i \leq k-1 
\nonumber \\
& \text{or } & 
1 + q^{j+2} \gamma_k^{(j+2)} (q) \quad \text{ if  } i = k,
\nonumber 
\end{eqnarray}
where
\[
\gamma^{(j)}_i (q) \in \mathbb{C}[[q]]. 
\]
This is our ``Empirical Hypothesis,'' in the sense explained in the
Introduction.

Using (\ref{Gjk-1i}) in the form (\ref{Grecursion}) below together
with (\ref{tautology}) below, we write each $G_i$, $ 1 \leq i \leq k$,
in terms of $G_{(k-1) j +1}, \dots, G_{(k-1) j + k}$ for each
$j=0,1,2,\dots$, giving a sequence of expressions (one for each $j$)
for $G_1, \dots, G_k$ of the form
\begin{eqnarray}\label{Gi}
G_i & = & 
{}_ih^{(j)}_1 G_{(k-1) j +1} + \cdots + {}_ih^{(j)}_k G_{(k-1) j + k},
\end{eqnarray}
where for each $j$ the coefficients ${}_ih^{(j)}_l$ form a $k \times
k$ matrix ${\bf h}^{(j)}$ of polynomials in $q$ with nonnegative
integral coefficients.  More explicitly, define row vectors
\[
{}_i {\bf h}^{(j)} = [{}_ih^{(j)}_1, \dots, {}_ih^{(j)}_k].
\]
For $j=0$ we have
\[
{}_i{\bf h}^{(0)} = [0,  \dots, 1, \dots, 0],
\]
with $1$ in the $i$-th position, so that ${\bf h}^{(0)}$ is the
identity matrix.  The ${}_i{\bf h}^{(j)}$ satisfy the same set of
recursions with respect to $j$, independently of $i$.  Explicitly:

\begin{proposition}\label{hrecursions}
Let $j \geq 1$.  With the left subscript ${}_i$ suppressed, we have
\begin{eqnarray}
h^{(j)}_1 & = &
h^{(j-1)}_1 + \cdots + h^{(j-1)}_{k-1} +  h^{(j-1)}_k 
\nonumber \\
h^{(j)}_2 & = & 
(h^{(j-1)}_1 + \cdots + h^{(j-1)}_{k-1}) q^j
\nonumber \\
& \cdots & 
\nonumber \\
h^{(j)}_{k-1} & = &
(h^{(j-1)}_1 + h^{(j-1)}_2) q^{(k-2)j}
\nonumber \\
h^{(j)}_k 
& = &
h^{(j-1)}_1 q^{(k-1)j}
\nonumber 
\end{eqnarray}
or in general, 
\begin{eqnarray}\label{hlj}
h^{(j)}_l & = &
(h^{(j-1)}_1 + \cdots + h^{(j-1)}_{k - l +1}) q^{ (l-1)j }, \quad 1 \leq l \leq k. 
\end{eqnarray}
In matrix form, this is:
\begin{eqnarray}\label{h=hA}
{\bf h}^{(j)} = {\bf h}^{(j-1)}{\bf A}_{(j)},
\end{eqnarray}
with $\bf h$ our $k \times k$ matrix defined above and with
\begin{eqnarray}\label{Aj}
{\bf A}_{(j)} = 
\left[ \begin{array}{ccccc}
1 & q^j & q^{2j} & \cdots  & q^{( k-1) j} \\
\vdots & \vdots &  \vdots & \swarrow  & \vdots \\
1 & q^j & q^{2j}  & \cdots & 0 \\
1 & q^j & 0 &  \cdots & 0 \\
1 & 0 & 0 & \cdots  & 0
\end{array} \right].
\end{eqnarray}
In particular,
\begin{eqnarray}\label{h=AAA}
{\bf h}^{(j)} = {\bf A}_{(1)}{\bf A}_{(2)} \cdots {\bf A}_{(j)}
\end{eqnarray}
for all $j \ge 0$.
\end{proposition}

\begin{proof}
By (\ref{Gjk-1i}),
\begin{eqnarray}\label{Grecursion}
G_{(k-1)(j-1) +l } & = &
G_{ (k-1) (j-1) + l+1 } + q^{ (k-l)j } G_{ (k-1)j + k - l +1 }
\end{eqnarray}
for $j \geq 1$, $l = 1, \dots, k-1$, and the lemma follows from the
repeated application of this formula together with the tautological
fact that
\begin{eqnarray}\label{tautology}
G_{(k-1)(j-1) + k } & = & G_{ (k-1) j + 1 }.
\end{eqnarray}
\end{proof}

In the course of the vertex-algebraic interpretation of the
Rogers-Selberg recursions in \cite{CLM2}, it was implicitly noticed
that matrices analogous to the matrices ${\bf A}_{(j)}$ along with
their inverses, involving the two variables in those recursions, could
be used to reformulate those recursions.  Such matrices indeed arise
naturally from recursions of these types.

\begin{proposition}\label{h=partitions}
For each $j \geq 1$ and $i,l=1,\dots,k$, the polynomial ${}_ih^{(j)}_l
\in \mathbb{C}[q]$ is the generating function for partitions with
difference at least 2 at distance $k-1$ such that $1$ appears at most
$k-i$ times, such that the largest part is at most $j$, and such that
$j$ appears exactly $l-1$ times.
\end{proposition}

\begin{proof}
It is sufficient to show that the combinatorial generating functions
described here have the same initial values and recursions as the
polynomials ${}_ih^{(j)}_l$.  We say that a partition is of type
$(k-1, k-i)$ if it has difference at least 2 at distance $k-1$ and 1
appears at most $k-i$ times.  Then (\ref{hlj}) corresponds to the
following combinatorial fact: For $j \ge 2$,
\begin{eqnarray}
\lefteqn{\text{the number of partitions of $m$ of type $(k-1, k-i)$ 
such that the largest part is at most}}
\nonumber \\
& & \text{$j$ and such that $j$ appears exactly $l-1$ times}
\nonumber \\
& = & 
\sum_{p=1}^{k-l+1}
\text{the number of partitions of $m-(l-1)j$ of type $(k-1, k-i)$ such
that the}
\nonumber \\
& & \quad
\text{largest part is at most $j-1$ and such that $j-1$ appears
exactly $p-1$ times.}
\nonumber 
\end{eqnarray}
The initial values
\[
{}_i{\bf h}^{(1)} = [1, q, q^2, \cdots, q^{k-i}, 0, \cdots, 0]
\]
also match those of the generating functions. 
\end{proof}

Recall the products $G_i$ in (\ref{Gidef}).

\begin{theorem}\label{Gi=difference}
For $1 \leq i \leq k$, $G_i$ is the generating function for partitions
with difference at least $2$ at distance $k-1$ such that $1$ appears
at most $k-i$ times.
\end{theorem}

\begin{proof}
This follows immediately from (\ref{Gi}), Proposition
\ref{h=partitions} and the Empirical Hypothesis.
\end{proof}

This result constitutes Gordon's identities, as formulated in Theorem
7.5 of \cite{A2}; the Rogers-Ramanujan identities form the special
case $k=2$.

\section{Matrix interpretation}
\setcounter{equation}{0}

The right-hand side of (\ref{Gi}) suggests a product of matrices, and
the recursions for the ${}_ih^{(j)}_l$ come from the recursions
(\ref{Grecursion}) (or equivalently, (\ref{Gjk-1i})) for the $G_s$, $s
\geq 1$, together with (\ref{tautology}).  We now express all of this
in matrix form, and in the process we quickly rederive Proposition
\ref{hrecursions}, obtaining the matrices ${\bf h}^{(j)}$ and their
properties from (\ref{Grecursion}) and (\ref{tautology}).

Set
\begin{eqnarray*}
{\bf G}_{(0)} = 
\left[ \begin{array}{c}
G_1 \\
\vdots \\
G_k
\end{array} \right]
\end{eqnarray*}
and in general,
\begin{eqnarray*}
{\bf G}_{(j)} = 
\left[ \begin{array}{c}
G_{(k-1)j+1} \\
\vdots \\
G_{(k-1)j+k}
\end{array} \right]
\end{eqnarray*}
for $j \geq 0$.  Also set
\begin{eqnarray*}
{\bf B}_{(j)} = 
\left[ \begin{array}{cccccc}
0           & 0            & \cdots   & 0       & 0        & 1       \\
0           & 0            & \cdots   & 0       & q^{-j}   & -q^{-j} \\
0           & 0            & \cdots   & q^{-2j} & -q^{-2j} & 0       \\
\vdots      & \vdots       & \swarrow & \vdots  & \vdots   & \vdots  \\
0           & q^{-(k-2)j}  & \cdots   & 0       & 0        & 0       \\
q^{-(k-1)j} & -q^{-(k-1)j} & \cdots   & 0       & 0        & 0
\end{array} \right]
\end{eqnarray*}
for $j \geq 1$.  Then (\ref{Gjk-1i})) (or equivalently,
(\ref{Grecursion})) and (\ref{tautology}) assert that
\begin{eqnarray}\label{G=BG}
{\bf G}_{(j)} = {\bf B}_{(j)}{\bf G}_{(j-1)}
\end{eqnarray}
for $j \geq 1$, so that
\[
{\bf G}_{(j)} = {\bf B}_{(j)}{\bf B}_{(j-1)} \cdots {\bf B}_{(1)}{\bf
G}_{(0)}
\]
for $j \geq 0$.  But
\[
{\bf B}_{(j)} = ({\bf A}_{(j)})^{-1}
\]
(recall (\ref{Aj})), which gives
\[
{\bf A}_{(j)}{\bf G}_{(j)} = {\bf G}_{(j-1)}
\]
for $j \geq 1$ and
\[
{\bf G}_{(0)} = {\bf A}_{(1)}{\bf A}_{(2)} \cdots {\bf A}_{(j)}{\bf
G}_{(j)}
\]
for $j \geq 0$.  Defining ${\bf h}^{(j)}$ recursively by
\[
{\bf h}^{(0)} = {\text{identity matrix,}}
\]
\[
{\bf h}^{(j)} = {\bf h}^{(j-1)}{\bf A}_{(j)}
\]
for $j \geq 1$ (cf. (\ref{h=hA})), we have (\ref{h=AAA}) along with
(\ref{Gi}), in the form
\[
{\bf G}_{(0)} = {\bf h}^{(j)}{\bf G}_{(j)}
\]
for each $j \geq 0$.  Thus from (\ref{Grecursion}) and
(\ref{tautology}) expressed in matrix form, we have an ``automatic''
reformulation and proof of Proposition \ref{hrecursions}, along with
(\ref{Gi}).

The combinatorial interpretation of the entries of ${\bf h}^{(j)}$ in
Proposition \ref{h=partitions} is a separate matter, as is the
combinatorial interpretation of the $G_s$, $s \geq 1$, given in
Theorem \ref{G=partitions} below.

\section{Interpretation of the sequence of expressions for $G_i$}
\setcounter{equation}{0}

All of the formal power series $G_s$ for $s \geq 1$ can be interpreted
as combinatorial generating functions, and this will allow us to
``explain'' the meaning of the sequence of equalities (\ref{Gi}) for
$G_1, \dots, G_k$.

We start with the expression of $G_1, \dots, G_k$ as the generating
functions given by Theorem \ref{Gi=difference}.  The recursions
(\ref{Grecursion}), or equivalently, (\ref{Gjk-1i}), determine all the
$G_s$ for $s \geq 1$ from $G_1, \dots, G_k$, and we know from Theorem
\ref{G} that all the $G_s$, $s \geq 1$, are formal power series.

We have the following ``complement'' to Proposition
\ref{h=partitions}, reflecting and illustrating the complementary
nature of the recursions (\ref{hlj}) and (\ref{Grecursion}) (or
equivalently, of (\ref{h=hA}) and (\ref{G=BG})):

\begin{theorem}\label{G=partitions}
For $j \geq 0$ and $l = 1, \dots, k$, the formal power series
$G_{(k-1)j+l}$ is the generating function for partitions with
difference at least 2 at distance $k-1$ such that the smallest part is
greater than $j$ and such that $j+1$ appears at most $k-l$ times.
\end{theorem}

\begin{proof}
It is sufficient to show that these combinatorial generating functions
have the same initial values and recursions as the formal power series
$G_{(k-1)j+l}$.  The recursion (\ref{Grecursion}) for $j \geq 1$ and
$l = 1, \dots, k-1$ corresponds to the following combinatorial fact:
For $j \ge 1$ and $l = 1, \dots, k-1$,
\begin{eqnarray}
\lefteqn{\text{the number of partitions of $m$ with difference at
least 2 at distance $k-1$ such that}}
\nonumber \\
& & \text{the smallest part is at least $j$ and such that $j$
appears exactly $k-l$ times }
\nonumber \\
& = & 
\text{the number of partitions of $m-(k-l)j$ with difference at least
2 at distance}
\nonumber \\
& & \quad
\text{$k-1$ such that the smallest part is greater than $j$ and such
that $j+1$ appears} 
\nonumber \\ 
& & \quad
\text{at most $l-1$ times.}
\nonumber 
\end{eqnarray}
By Theorem \ref{Gi=difference} (the case $j=0$), these assertions
prove the result.  (Note that the combinatorial interpretations of the
two expressions equated in (\ref{tautology}) indeed agree.)
\end{proof}

For $j=0$, (\ref{Gi}) says simply that $G_i = G_i$, and for $j \geq
1$, combining Proposition \ref{h=partitions} and Theorem
\ref{G=partitions} we immediately have:

\begin{theorem}\label{interp}
For $l = 1, \dots, k$ and $j \geq 1$, the right-hand side of
(\ref{Gi}) expresses the generating function $G_i$ as the sum of its
contributions corresponding to the number of times, namely, $0, 1,
\dots, k-1$, that the part $j$ appears in a partition.
\end{theorem}

\begin{remark}
{\em With the benefit of the picture that has emerged, we can give an
alternate, shorter proof of Theorem \ref{Gi=difference} (Gordon's
identities), without needing Proposition \ref{hrecursions} or
\ref{h=partitions}, using the following uniqueness observation: Let
$J_1, J_2, \dots$ be a sequence of formal power series in $q$ with
constant term $1$ satisfying the recursions (\ref{Grecursion}) for $j
\geq 1$ and $l = 1, \dots, k-1$ (with $J$ in place of $G$), and
suppose that the Empirical Hypothesis holds for $J_1, J_2, \dots$.
The comments above in connection with (\ref{Gi}) give a sequence of
expressions of the form (\ref{Gi}) (with $J$ in place of $G$), with
the coefficients the same polynomials ${}_ih^{(j)}_l$ as in (\ref{Gi})
(and now, we do not have to compute them).  By the Empirical
Hypothesis, the $k$ formal power series $J_1, \dots, J_k$ are uniquely
determined, and thus so is the whole sequence $J_1, J_2, \dots$.  But
by the proof of Theorem \ref{G=partitions}, the combinatorial
generating functions defined in the statement of Theorem
\ref{G=partitions} form a sequence $K_1, K_2, \dots$ of formal power
series with constant term $1$ satisfying the recursions
(\ref{Grecursion}), and the Empirical Hypothesis trivially holds for
$K_1, K_2, \dots$.  Thus by the uniqueness, $J_s = K_s$ for each $s
\geq 1$.  Then by Theorem \ref{G}, which gives the Empirical
Hypothesis for $G_1, G_2, \dots$, we have (without using Theorem
\ref{Gi=difference}) that $G_s = K_s$ for each $s \geq 1$, and this
statement for $s = 1, \dots, k$ constitutes Gordon's identities.  This
remark generalizes the corresponding alternate proof of the
Rogers-Ramanujan identities discsussed in \cite{AB}, \cite{R} and
\cite{A3}.}

\end{remark}

\vspace{.3in}

\noindent {\small \sc Department of Mathematics, Rutgers University,
Piscataway, NJ 08854} \\ {\em E--mail address}:
lepowsky@math.rutgers.edu \\

\vspace{.1in}

\noindent {\small \sc Department of Mathematics, Rutgers University,
Piscataway, NJ 08854} \\
\noindent Current address:\\
\noindent{\small \sc 
Mathematical Sciences Center, Tsinghua University, Beijing, China 100084}\\
{\em E--mail address}:
mxzhu@math.tsinghua.edu.cn \\

\end{document}